\pdfoutput=1
\documentclass[a4paper]{amsart}
\usepackage[T1]{fontenc}
\usepackage{mathrsfs}
\usepackage{amstext}
\usepackage{amsthm}
\usepackage{amssymb}
\usepackage{stackrel}
\usepackage[unicode=true,pdfusetitle,
 bookmarks=true,bookmarksnumbered=false,bookmarksopen=false,
 breaklinks=false,pdfborder={0 0 0},pdfborderstyle={},backref=false,colorlinks=false]
 {hyperref}

\makeatletter

\numberwithin{equation}{section}
\numberwithin{figure}{section}
\theoremstyle{plain}
\newtheorem{thm}{\protect\theoremname}[section]
\theoremstyle{definition}
\newtheorem{defn}[thm]{\protect\definitionname}
\theoremstyle{plain}
\newtheorem{cor}[thm]{\protect\corollaryname}
\theoremstyle{definition}
\newtheorem{example}[thm]{\protect\examplename}
\theoremstyle{plain}
\newtheorem{prop}[thm]{\protect\propositionname}
\theoremstyle{plain}
\newtheorem{qst}[thm]{\protect\questionname}
\theoremstyle{definition}
\newtheorem{notation}[thm]{\protect\notationname}
\theoremstyle{remark}
\newtheorem{rem}[thm]{\protect\remarkname}
\theoremstyle{plain}
\newtheorem{lem}[thm]{\protect\lemmaname}

\usepackage{tikz-cd}
\usepackage{mathtools}
\usepackage{adjustbox}
\usepackage{enumitem}
\tikzcdset{scale cd/.style={every label/.append style={scale=#1},
    cells={nodes={scale=#1}}}}
\usepackage{quiver}
\usepackage{eucal}
\usepackage{mleftright}
\mleftright

\makeatother

\providecommand{\corollaryname}{Corollary}
\providecommand{\definitionname}{Definition}
\providecommand{\examplename}{Example}
\providecommand{\lemmaname}{Lemma}
\providecommand{\notationname}{Notation}
\providecommand{\propositionname}{Proposition}
\providecommand{\questionname}{Question}
\providecommand{\remarkname}{Remark}
\providecommand{\theoremname}{Theorem}

\begin{document}
\global\long\def\sf#1{\mathsf{#1}}%

\global\long\def\scr#1{\mathscr{{#1}}}%

\global\long\def\cal#1{\mathcal{#1}}%

\global\long\def\bb#1{\mathbb{#1}}%

\global\long\def\bf#1{\mathbf{#1}}%

\global\long\def\frak#1{\mathfrak{#1}}%

\global\long\def\fr#1{\mathfrak{#1}}%

\global\long\def\u#1{\underline{#1}}%

\global\long\def\tild#1{\widetilde{#1}}%

\global\long\def\mrm#1{\mathrm{#1}}%

\global\long\def\pr#1{\left(#1\right)}%

\global\long\def\abs#1{\left|#1\right|}%

\global\long\def\inp#1{\left\langle #1\right\rangle }%

\global\long\def\br#1{\left\{  #1\right\}  }%

\global\long\def\norm#1{\left\Vert #1\right\Vert }%

\global\long\def\hat#1{\widehat{#1}}%

\global\long\def\opn#1{\operatorname{#1}}%

\global\long\def\bigmid{\,\middle|\,}%

\global\long\def\Top{\sf{Top}}%

\global\long\def\Set{\sf{Set}}%

\global\long\def\SS{\sf{sSet}}%

\global\long\def\BS{\sf{bsSet}}%

\global\long\def\Kan{\sf{Kan}}%

\global\long\def\Cat{\mathcal{C}\sf{at}}%

\global\long\def\Grpd{\mathcal{G}\sf{rpd}}%

\global\long\def\Res{\mathcal{R}\sf{es}}%

\global\long\def\imfld{\cal M\mathsf{fld}}%

\global\long\def\ids{\cal D\sf{isk}}%

\global\long\def\ich{\cal C\sf h}%

\global\long\def\SW{\mathcal{SW}}%

\global\long\def\SHC{\mathcal{SHC}}%

\global\long\def\Fib{\mathcal{F}\mathsf{ib}}%

\global\long\def\Bund{\mathcal{B}\mathsf{und}}%

\global\long\def\Fam{\cal F\sf{amOp}}%

\global\long\def\B{\sf B}%

\global\long\def\Spaces{\sf{Spaces}}%

\global\long\def\Mod{\sf{Mod}}%

\global\long\def\Nec{\sf{Nec}}%

\global\long\def\Fin{\sf{Fin}}%

\global\long\def\Ch{\sf{Ch}}%

\global\long\def\Ab{\sf{Ab}}%

\global\long\def\SA{\sf{sAb}}%

\global\long\def\P{\mathsf{POp}}%

\global\long\def\Op{\mathcal{O}\mathsf{p}}%

\global\long\def\Opg{\mathcal{O}\mathsf{p}_{\infty}^{\mathrm{gn}}}%

\global\long\def\Tup{\mathsf{Tup}}%

\global\long\def\H{\cal H}%

\global\long\def\Mfld{\cal M\mathsf{fld}}%

\global\long\def\Disk{\cal D\mathsf{isk}}%

\global\long\def\Acc{\mathcal{A}\mathsf{cc}}%

\global\long\def\Pr{\mathcal{P}\mathrm{\mathsf{r}}}%

\global\long\def\Del{\mathbf{\Delta}}%

\global\long\def\id{\operatorname{id}}%

\global\long\def\Aut{\operatorname{Aut}}%

\global\long\def\End{\operatorname{End}}%

\global\long\def\Hom{\operatorname{Hom}}%

\global\long\def\Ext{\operatorname{Ext}}%

\global\long\def\sk{\operatorname{sk}}%

\global\long\def\ihom{\underline{\operatorname{Hom}}}%

\global\long\def\N{\mathrm{N}}%

\global\long\def\-{\text{-}}%

\global\long\def\op{\mathrm{op}}%

\global\long\def\To{\Rightarrow}%

\global\long\def\rr{\rightrightarrows}%

\global\long\def\rl{\rightleftarrows}%

\global\long\def\mono{\rightarrowtail}%

\global\long\def\epi{\twoheadrightarrow}%

\global\long\def\comma{\downarrow}%

\global\long\def\ot{\leftarrow}%

\global\long\def\corr{\leftrightsquigarrow}%

\global\long\def\lim{\operatorname{lim}}%

\global\long\def\colim{\operatorname{colim}}%

\global\long\def\holim{\operatorname{holim}}%

\global\long\def\hocolim{\operatorname{hocolim}}%

\global\long\def\Ran{\operatorname{Ran}}%

\global\long\def\Lan{\operatorname{Lan}}%

\global\long\def\Sk{\operatorname{Sk}}%

\global\long\def\Sd{\operatorname{Sd}}%

\global\long\def\Ex{\operatorname{Ex}}%

\global\long\def\Cosk{\operatorname{Cosk}}%

\global\long\def\Sing{\operatorname{Sing}}%

\global\long\def\Sp{\operatorname{Sp}}%

\global\long\def\Spc{\operatorname{Spc}}%

\global\long\def\Fun{\operatorname{Fun}}%

\global\long\def\map{\operatorname{map}}%

\global\long\def\diag{\operatorname{diag}}%

\global\long\def\Gap{\operatorname{Gap}}%

\global\long\def\cc{\operatorname{cc}}%

\global\long\def\ob{\operatorname{ob}}%

\global\long\def\Map{\operatorname{Map}}%

\global\long\def\Rfib{\operatorname{RFib}}%

\global\long\def\Lfib{\operatorname{LFib}}%

\global\long\def\Tw{\operatorname{Tw}}%

\global\long\def\Equiv{\operatorname{Equiv}}%

\global\long\def\Arr{\operatorname{Arr}}%

\global\long\def\Cyl{\operatorname{Cyl}}%

\global\long\def\Path{\operatorname{Path}}%

\global\long\def\Alg{\operatorname{Alg}}%

\global\long\def\ho{\operatorname{ho}}%

\global\long\def\Comm{\operatorname{Comm}}%

\global\long\def\Triv{\operatorname{Triv}}%

\global\long\def\triv{\operatorname{triv}}%

\global\long\def\Env{\operatorname{Env}}%

\global\long\def\Act{\operatorname{Act}}%

\global\long\def\loc{\operatorname{loc}}%

\global\long\def\Assem{\operatorname{Assem}}%

\global\long\def\Nat{\operatorname{Nat}}%

\global\long\def\Conf{\operatorname{Conf}}%

\global\long\def\Rect{\operatorname{Rect}}%

\global\long\def\Emb{\operatorname{Emb}}%

\global\long\def\Homeo{\operatorname{Homeo}}%

\global\long\def\mor{\operatorname{mor}}%

\global\long\def\Germ{\operatorname{Germ}}%

\global\long\def\Post{\operatorname{Post}}%

\global\long\def\Sub{\operatorname{Sub}}%

\global\long\def\Shv{\operatorname{Shv}}%

\global\long\def\Cov{\operatorname{Cov}}%

\global\long\def\Disc{\operatorname{Disc}}%

\global\long\def\Open{\operatorname{Open}}%

\global\long\def\Unm{\operatorname{Unm}}%

\global\long\def\hoeq{\mathrm{hoeq}}%

\global\long\def\CSS{\mathrm{CSS}}%

\global\long\def\rel{\mathrm{rel}}%

\global\long\def\Reedy{\mathrm{Reedy}}%

\global\long\def\cart{\mathrm{cart}}%

\global\long\def\disj{\mathrm{disj}}%

\global\long\def\Top{\mathrm{Top}}%

\global\long\def\lax{\mathrm{lax}}%

\global\long\def\weq{\mathrm{weq}}%

\global\long\def\fib{\mathrm{fib}}%

\global\long\def\inert{\mathrm{inert}}%

\global\long\def\act{\mathrm{act}}%

\global\long\def\cof{\mathrm{cof}}%

\global\long\def\inj{\mathrm{inj}}%

\global\long\def\univ{\mathrm{univ}}%

\global\long\def\Ker{\opn{Ker}}%

\global\long\def\Coker{\opn{Coker}}%

\global\long\def\Im{\opn{Im}}%

\global\long\def\Coim{\opn{Im}}%

\global\long\def\coker{\opn{coker}}%

\global\long\def\im{\opn{\mathrm{im}}}%

\global\long\def\coim{\opn{coim}}%

\global\long\def\gn{\mathrm{gn}}%

\global\long\def\Mon{\opn{Mon}}%

\global\long\def\Un{\opn{Un}}%

\global\long\def\St{\opn{St}}%

\global\long\def\cun{\widetilde{\opn{Un}}}%

\global\long\def\cst{\widetilde{\opn{St}}}%

\global\long\def\Sym{\operatorname{Sym}}%

\global\long\def\CA{\operatorname{CAlg}}%

\global\long\def\Ind{\operatorname{Ind}}%

\global\long\def\rd{\mathrm{rd}}%

\global\long\def\xmono#1#2{\stackrel[#2]{#1}{\rightarrowtail}}%

\global\long\def\xepi#1#2{\stackrel[#2]{#1}{\twoheadrightarrow}}%

\global\long\def\adj{\stackrel[\longleftarrow]{\longrightarrow}{\bot}}%

\global\long\def\btimes{\boxtimes}%

\global\long\def\ps#1#2{\prescript{}{#1}{#2}}%

\global\long\def\ups#1#2{\prescript{#1}{}{#2}}%

\global\long\def\hofib{\mathrm{hofib}}%

\global\long\def\cofib{\mathrm{cofib}}%

\global\long\def\Vee{\bigvee}%

\global\long\def\w{\wedge}%

\global\long\def\t{\otimes}%

\global\long\def\bp{\boxplus}%

\global\long\def\rcone{\triangleright}%

\global\long\def\lcone{\triangleleft}%

\global\long\def\p{\prime}%

\global\long\def\pp{\prime\prime}%

\global\long\def\W{\overline{W}}%

\global\long\def\o#1{\overline{#1}}%

\global\long\def\fp{\overrightarrow{\times}}%

\title{Classification Diagrams of Marked Simplicial Sets}
\begin{abstract}
We prove that the classification diagram functor from the category
of marked simplicial sets to the category of bisimplicial sets carries
cartesian equivalences to Rezk equivalences. As a corollary, we obtain
Mazel-Gee's theorem on localizations of relative $\infty$-categories.
\end{abstract}

\author{Kensuke Arakawa}
\email{arakawa.kensuke.22c@st.kyoto-u.ac.jp}
\address{Department of Mathematics, Kyoto University, Kyoto, 606-8502, Japan}

\maketitle
\tableofcontents{}

\section{Introduction}

\subsection{Localizations of $\infty$-Categories}

The process of freely inverting morphisms in a given category, known
as localization, is fundamental to category theory. Localizations
of categories can easily be generalized to the setting of $\infty$-categories\footnote{By $\infty$-categories, we mean quasi-categories of \cite{Joyal_qcat_Kan}.
See \cite{Landoo-cat} for a friendly introduction of $\infty$-categories.}.
\begin{defn}
\cite[Definition 2.4.2]{Landoo-cat}\label{def:loc} Let $\cal C$
and $\cal D$ be $\infty$-categories and let $S$ be a set of morphisms
of $\cal C$ containing all identity morphisms. A functor $L:\cal C\to\cal D$
of $\infty$-categories is said to \textbf{exhibit $\cal D$ as a
(Dwyer--Kan) localization of $\cal C$ with respect to $S$} if it
satisfies the following conditions:
\begin{itemize}
\item The functor $L$ carries every morphism in $S$ to an equivalence.
\item For every $\infty$-category $\cal E$, the functor
\[
\Fun\pr{\cal D,\cal E}\to\Fun^{S}\pr{\cal C,\cal E}
\]
is a categorical equivalence, where $\Fun^{S}\pr{\cal C,\cal E}$
denotes the full subcategory of $\Fun\pr{\cal C,\cal E}$ spanned
by the functors $\cal C\to\cal E$ carrying every morphism in $S$
to an equivalence.
\end{itemize}
If $L$ satisfies these conditions, we call $\cal D$ the\textbf{
localization of $\cal C$ with respect to $S$} and write $\cal D=\cal C[S^{-1}]$.
\end{defn}

It is often helpful to \textit{know} that a certain functor $\cal C\to\cal D$
of $\infty$-categories exhibits $\cal D$ as a localization of $\cal C$
with respect to a set $S$ of morphisms. For one thing, it helps advance
the understanding of $\cal D$, for it characterizes $\cal D$ by
a universal property. For another, it often allows us to sweep the
problem of coherency under the rug. Constructing a functor of the
form $\cal D\to\cal E$ can sometimes be hopelessly hard because of
the immense amount of data one must specify. But if $\cal C$ is sufficiently
nice, then it might be feasible to construct a functor $f:\cal C\to\cal E$.
And if $f$ carries each morphism in $S$ to an equivalence, we obtain
the desired functor $\cal D\to\cal E$ for free, using the universal
property of localizations. 

In contrast, it is often hard to \textit{show} that an $\infty$-category
is a localization of another. As such, it is of central concern in
modern homotopy theory to present complicated $\infty$-categories
as localizations of simpler $\infty$-categories; sometimes such a
presentation can even be an end in itself. See, for instance, \cite[Theorem 1.3.4.20]{HA},
\cite[$\S$2.4]{AF_FHTM}, and \cite{LST2022}.

\subsection{Mazel-Gee's Localization Theorem}

In the previous subsection, we saw that localizations of $\infty$-categories
are useful but are often inaccessible. To remedy this situation, Mazel-Gee
established a convenient criterion for localizations. 

To explain Mazel-Gee's work, we must recall the definition of another
model of $\pr{\infty,1}$-categories, namely, complete Segal spaces.
A \textbf{complete Segal space} is a bisimplicial set $X$ whose $n$th
column $X_{n,\ast}$ is a Kan complex modeling the space of $n$ composable
arrows of the $\pr{\infty,1}$-category $X$ presents. (See \cite[$\S$6]{Rezk01}
for a precise definition.) For example, if $\cal C$ is an $\infty$-category,
then the bisimplicial set $N\pr{\cal C}$ whose $n$th column is given
by the maximal sub Kan complex of $\Fun\pr{\Delta^{n},\cal C}$ is
a complete Segal space. Rezk, the inventor of complete Segal spaces,
constructed a model structure for complete Segal spaces on the category
$\BS$ of bisimplicail sets \cite[Theorem 7.2]{Rezk01}. We will denote
this model structure by $\BS_{\CSS}$ and call its weak equivalences
\textbf{Rezk equivalences}.

Now suppose we are given an $\infty$-category $\cal C$ and a subcategory
$\cal W\subset\cal C$. If we want to formally invert the morphisms
in $\cal W$, one thing we could try is to formally replace equivalences
with $\cal W$ in the definition of $N\pr{\cal C}$. In other words,
we consider the bisimplicial set $N\pr{\cal C,\cal W}$ whose $n$th
column is given by the fiber product
\[
\Fun\pr{\Delta^{n},\cal C}\times_{\cal C^{n+1}}\cal W^{n+1}.
\]
We call the bisimplicial set $N\pr{\cal C,\cal W}$ the \textbf{classification
diagram} of the pair $\pr{\cal C,\cal W}$. Of course, the classification
diagram is no longer a complete Segal space; its columns may not even
be Kan complexes, for $\cal W$ may contain non-equivalences. Nonetheless,
Mazel-Gee's theorem asserts that this (perhaps na\"ive) construction
does compute localizations\footnote{We should remark that a germ of Theorem \ref{thm:MGloc} was already
present in the paper \cite{Rezk01}, in which Rezk introduced complete
Segal spaces. See \cite[Theorem 8.3]{Rezk01}.}:
\begin{thm}
[Mazel-Gee's Localization Theorem {\cite[Theorem 3.8]{MR4045352}}]\label{thm:MGloc}Let
$\cal C$ and $\cal D$ be $\infty$-categories, let $\cal W\subset\cal C$
be a subcategory containing all equivalences, and let $f:\cal C\to\cal D$
be a functor which carries every morphism in $\cal W$ to an equivalence.
The following conditions are equivalent:
\begin{enumerate}
\item The map $N\pr{\cal C,\cal W}\to N\pr{\cal D}$ is a Rezk equivalence.
\item The functor $f$ exhibits $\cal D$ as a localization of $\cal C$
with respect to the morphisms in $\cal W$.
\end{enumerate}
\end{thm}

What is great about Theorem \ref{thm:MGloc} is that it gives us a
sufficient condition for a functor to be a localization functor. For
example, it is known that column-wise weak homotopy equivalences and
row-wise categorical equivalences of bisimplicial sets are Rezk equivalences
(\cite[Theorem 7.2]{Rezk01}, \cite[Theorem 4.5]{JT07}). Therefore,
Theorem \ref{thm:MGloc} implies:
\begin{cor}
\label{cor:MG_loc}Let $\cal C$ and $\cal D$ be $\infty$-categories,
let $\cal W\subset\cal C$ be a subcategory containing all equivalences,
and let $f:\cal C\to\cal D$ be a functor which carries every morphism
in $\cal W$ to an equivalence. Suppose that the map $N\pr{\cal C,\cal W}\to N\pr{\cal D}$
is either a column-wise weak homotopy equivalence or a row-wise categorical
equivalence. Then $f$ exhibits $\cal D$ as a localization of $\cal C$
with respect to the morphisms in $\cal W$. 
\end{cor}

\subsection{What This Paper is About}

The goal of this paper is to illuminate and generalize Mazel-Gee's
localization Theorem (Theorem \ref{thm:MGloc}) by using marked simplicial
sets.

Recall that a \textbf{marked simplicial set} is a pair $\pr{X,S}$,
where $X$ is a simplicial set and $S$ is a set of edges of $X$
containing all degenerate edges; a morphism of marked simplicial sets
$\pr{X,S}\to\pr{Y,T}$ is just a morphism of simplicial sets $X\to Y$
which carries $S$ into $T$. Here are some examples of marked simplicial
sets:
\begin{example}
\hfill
\begin{enumerate}
\item Let $\cal C$ be an $\infty$-category. Then the pair $\cal C^{\natural}=\pr{\cal C,\{\text{equivalences of }\cal C\}}$
is a marked simplicial set.
\item If $X$ is a simplicial set, then the pair $X^{\flat}=\pr{X,\{\text{degenerate edges of }X\}}$
is a marked simplicial set.
\item If $X$ is a simplicial set, then $X^{\sharp}=\pr{X,X_{1}}$ is a
marked simplicial set.
\end{enumerate}
\end{example}

There is a special class of morphisms of marked simplicial sets which
is closely related to localizations. A morphism of marked simplicial
sets $f:\pr{X,S}\to\pr{Y,T}$ is called a \textbf{cartesian equivalence}\footnote{The use of the adjective ``cartesian'' is explained by the fact
that cartesian equivalences were first considered in the study of
\textit{cartesian fibrations} of simplicial sets. See \cite[$\S$3.1]{HTT}
for more details.}\textbf{ }if for each $\infty$-category $\cal C$, the functor
\[
\Fun^{T}\pr{Y,\cal C}\to\Fun^{S}\pr{X,\cal C}
\]
is a categorical equivalence. Here $\Fun^{S}\pr{X,\cal C}$ denotes
the full subcategory of $\Fun\pr{X,\cal C}$ consisting of the diagrams
$X\to\cal C$ which carries every morphism in $S$ to an equivalence
of $\cal C$. Thus, intuitively, $f$ is a cartesian equivalence if
and only if it induces a categorical equivalence after localizing
$X$ and $Y$ with respect to $S$ and $T$. In particular, we can
reformulate the definition of localizations of $\infty$-categories
by using cartesian equivalences:
\begin{prop}
\label{prop:intro_1}Let $\cal C$ and $\cal D$ be $\infty$-categories,
let $S$ be a set of morphisms of $\cal C$ containing all identity
morphisms, and let $f:\cal C\to\cal D$ be a functor which carries
every morphism in $S$ to an equivalence. The following conditions
are equivalent:
\begin{enumerate}
\item The map $\pr{\cal C,S}\to\cal D^{\natural}$ is a cartesian equivalence.
\item The functor $f$ exhibits $\cal D$ as a localization of $\cal C$
with respect to $S$.
\end{enumerate}
\end{prop}

There is even a model structure for cartesian equivalences, called
the \textbf{cartesian model structure} \cite[$\S$3.1]{HTT}. In this
model structure, the cofibrations are the monomorphisms, the weak
equivalences are the cartesian equivalences, and the fibrant objects
are the marked simplicial sets of the form $\cal C^{\natural}$, where
$\cal C$ is an $\infty$-category. We denote this model structure
by $\SS_{{\rm cart}}^{+}$

There is a natural extension of the classification diagram construction
in the setting of marked simplicial sets. Define the \textbf{classification
diagram} functor $N:\SS^{+}\to\BS$ by setting
\[
N\pr{X,S}_{n,m}=\{\text{maps }\pr{\Delta^{n}}^{\flat}\times\pr{\Delta^{m}}^{\sharp}\to\pr{X,S}\text{ of marked simplicial sets}\}.
\]
If $\cal C$ is an $\infty$-category and $\cal W\subset\cal C$ is
its subcategory, then $N\pr{\cal C,\cal W_{1}}$ is nothing but the
classification diagram of the pair $\pr{\cal C,\cal W}$ defined in
the previous subsection. In light of this and Proposition \ref{prop:intro_1},
we can interpret Theorem \ref{thm:MGloc} as saying that:
\begin{quote}
The classification diagram functor $N:\SS_{{\rm cart}}^{+}\to\BS_{\CSS}$
preserves and reflects \textit{some} weak equivalences between possibly
\textit{non-fibrant} objects.
\end{quote}

This is curious. It is not hard to see that the functor $N:\SS_{{\rm cart}}^{+}\to\BS_{\CSS}$
is a right Quillen equivalence (Theorem \ref{thm:relations}), so
that it preserves and reflects weak equivalences of fibrant objects.
However, right Quillen equivalences often fail to preserve and reflect
weak equivalences of non-fibrant objects. And there is a good reason
for this: The whole point of using model categories is to consider
the homotopy theory of good (i.e., cofibrant-fibrant) objects, so
we should not expect right Quillen equivalences to keep track of weak
equivalences of non-fibrant objects. Nonetheless, Mazel-Gee's Theorem
says that the classification diagram functor does preserve and reflect
some weak equivalences. This naturally leads us to the following question:
\begin{qst}
\label{que:intro}\hfill
\begin{enumerate}
\item Which cartesian equivalences does the classification diagram functor
preserve? 
\item Are there non-cartesian equivalences that induce Rezk equivalences
between the classification diagrams?
\end{enumerate}
\end{qst}

The following Theorem, which is the main result of this paper, gives
a complete answer to Question \ref{que:intro}. 
\begin{thm}
[Theorem \ref{thm:main}]\label{thm:intro_2}Let $f:\pr{X,S}\to\pr{Y,T}$
be a morphism of marked simplicial sets. The following conditions
are equivalent:
\begin{enumerate}
\item The map $f$ is a cartesian equivalence.
\item The map $N\pr f$ is a Rezk equivalence.
\end{enumerate}
\end{thm}

To understand how delicate Theorem \ref{thm:intro_2} is, it is instructive
to consider the forgetful functor $\SS^{+}\to\SS$. Like the classification
diagram functor, this functor is a right Quillen equivalence (with
respect to the cartesian model structure and the model structure for
quasi-categories \cite[$\S$2.2.5]{HTT}, also known as the Joyal model
structure) and takes values in a category which does not remember
markings. However, there are plenty of cartesian equivalences whose
underlying map of simplicial sets is not a weak categorical equivalence;
the inclusion $\pr{\Lambda_{2}^{2}}^{\flat}\cup\pr{\Delta^{\{1,2\}}}^{\sharp}\to\pr{\Delta^{2}}^{\flat}\cup\pr{\Delta^{\{1,2\}}}^{\sharp}$
is one such example. So what Theorem \ref{thm:intro_2} says is that,
somehow, the classification diagram functor magically remembers all
information about markings while disposing of markings.

It is also worth mentioning that our proof of Theorem \ref{thm:MGloc}
is rather short and elementary. When Mazel-Gee proved Theorem \ref{thm:MGloc}
in \cite{MR4045352}, he remarked that ``our proof...is (perhaps
unexpectedly, and certainly unsatisfyingly) complicated,'' leaving
the possibility of a more simple proof. This is understandable, for
his argument relies on a lemma whose proof lasts for pages. We hope
that our proof of Theorem \ref{thm:intro_2} (which is strictly stronger
than Theorem \ref{thm:MGloc}) will be a good alternative to Mazel-Gee's
proof of Theorem \ref{thm:MGloc}.

\subsection{Outline of the Paper}

In Section \ref{sec:mCSS}, we construct a model structure of \textit{marked}
complete Segal spaces, which is an analog of the cartesian model structure
on the category of \textit{marked bisimplicial sets}. The appearance
of marking comes as no surprise, for localizations of $\infty$-categories
are essentially fibrant replacements in the category of marked simplicial
sets. In Section \ref{sec:relations}, we relate the model structure
of marked complete Segal spaces with various other model structures.
Based on the contents of these sections, we prove our main result
(Theorem \ref{thm:main}) in Section \ref{sec:localization_thm}. 

Readers only interested in Theorem \ref{thm:intro_2} may skip most
of Sections \ref{sec:mCSS} and \ref{sec:relations}, but make sure
to understand the main results of these sections, which are Theorem
\ref{thm:mCSS_model_structure} and Theorem \ref{thm:relations}.

\subsection{A Remark on Marked Complete Segal Spaces}

The existence of the model structure for marked complete Segal spaces
was first announced by Rasekh in \cite[Theorem 2.25]{Rasekh21}. However,
his proof contains an error and only produces a model structure with
non-cofibrant objects, contrary to what he claims. We resolve this
issue in Section \ref{sec:mCSS}, where we construct the model structure
by a method different from Rasekh's.

\subsection*{Acknowledgment}

I appreciate David Ayala for informing me of the result of Mazel-Gee,
and Hiro Lee Tanaka for introducing David Ayala to me. I am also indebted
to Daisuke Kishimoto and Mitsunobu Tsutaya for their constant support
and encouragement. 

\subsection*{Notation and Terminology}

In this paper, we adopt the following conventions for notation and
terminology.
\begin{itemize}
\item The symbols $\SS$, $\SS^{+}$, $\BS$ denote the categories of simplicial
sets, marked simplicial sets, and bisimplicial sets, respectively.
When categories are equipped with model structures, we indicate it
by writing the name of the model structure in the subscript. Examples
include the \textbf{Joyal model structure} \cite[$\S$2.2.5]{HTT},
denoted by $\SS_{{\rm Joyal}}$; the \textbf{cartesian model structure}
\cite[$\S$3.1.3]{HTT}, denoted by $\SS_{{\rm cart}}^{+}$; and the
\textbf{model structure for complete Segal spaces} \cite[Theorem 7.2]{Rezk01},
denoted by $\BS_{\CSS}$. 
\item If $\cal M$ is a cartesian closed category, a model structure on
$\cal M$ is said to be \textbf{cartesian} if its model structure
is monoidal \cite[Definition A.3.1.2]{HTT} with respect to binary
product.
\item We will refer to the fibrant objects of the Joyal model structure
as\textbf{ $\infty$-categories}. Usually, $\infty$-categories are
denoted by a calligraphic letter, such as $\cal C$. If $\cal C$
is an $\infty$-category, its \textbf{core} $\cal C^{\simeq}\subset\cal C$
is the largest Kan complex contained in $\cal C$.
\item Like $\infty$-categories, complete Segal spaces will often be denoted
by calligraphic letters.
\item The symbol $J$ denotes the nerve of the groupoid with two objects
$0$ and $1$ and exactly one morphism between any two objects.
\item A morphism of simplicial sets is called a \textbf{trivial fibration}
if it has the right lifting property for all monomorphisms of simplicial
sets.
\item A marked simplicial set is usually denoted by an alphabet with a line
above it, such as $\overline{X}$. Its underlying simplicial set is
typically denoted by removing the line. For example, if $\overline{X}$
is a marked simplicial set, its underlying simplicial set is $X$.
\item A class of morphisms in a category with small colimits is said to
be\textbf{ saturated} if it is stable under pushouts, transfinite
compositions, and retracts.
\item Other notations and terminology follow \cite{HTT}. Some non-standard
notation and terminology will be introduced where appropriate.
\end{itemize}

\section{\label{sec:mCSS}A Model Structure for Marked Complete Segal Spaces}

The subject of this section is marked bisimplicial sets. Marked bisimplicial
sets are to complete Segal spaces what marked simplicial sets are
to $\infty$-categories, and they were first introduced by Rasekh
in \cite{Rasekh21}. The goal of this section is to construct a model
structure on the category of marked bisimplicial sets whose fibrant
objects are the complete Segal spaces whose equivalences are marked.

We start by recalling the definition of marked bisimplicial sets.

\begin{defn}
\cite[Definition 2.2]{Rasekh21} A \textbf{marked bisimplicial set}
is a pair $\pr{X,S}$, where $X$ is a bisimplicial set and $S\subset X_{1}=X_{1,\ast}$
is a simplicial subset of the first column of $X$ which contains
the image of the map $X_{0}\to X_{1}$. A \textbf{morphism} of marked
bisimplicial sets $\pr{X,S}\to\pr{Y,T}$ is a morphism of bisimplicial
sets $X\to Y$ which carries $S$ into $T$. Marked bisimplicial sets
and their morphisms form a category, which we denote by $\BS^{+}$. 
\end{defn}

\begin{notation}
We define functors $\pr -^{\flat},\pr -^{\sharp}:\BS\to\BS^{+}$ by
$X^{\flat}=\pr{X,X_{0}}$, $X^{\sharp}=\pr{X,X_{1}}$. We also define
a functor $\Unm:\BS^{+}\to\BS$ by $\Unm\pr{X,S}=X$. There is an
adjunction $\pr -^{\flat}\dashv\Unm\dashv\pr -^{\sharp}$.
\end{notation}

Just like we did so for marked simplicial sets, we will typically
denote a marked simplicial set by putting an overline above its underlying
bisimplicial set.
\begin{example}
Recall that if $X$ and $Y$ are simplicial sets, their \textbf{box
product} $X\btimes Y$ is the bisimplicial set given by $\pr{X\btimes Y}_{m,n}=X_{m}\times Y_{n}$.
There is an analog of this construction in the marked setting: If
$\pr{X,S}$ is a marked simplicial set and $Y$ is a simplicial set,
we define their \textbf{box product} $\pr{X,S}\btimes Y$ by 
\[
\pr{X,S}\btimes Y=\pr{X\btimes Y,S\times Y}.
\]
Note that box products commute with flat and sharp, in the sense that
if $X$ and $Y$ are simplicial sets, then $X^{\flat}\btimes Y=\pr{X\btimes Y}^{\flat}$
and $X^{\sharp}\btimes Y=\pr{X\btimes Y}^{\sharp}$.
\end{example}

\begin{example}
If $\cal C$ is a complete Segal space, we let $\cal C^{\natural}$
denote the marked bisimplicial set $\pr{\cal C,\cal C_{\hoeq}}$,
where $\cal C_{\hoeq}\subset\cal C_{1}$ is the full simplicial subset
spanned by the homotopy equivalences of $\cal C$ \cite[$\S$5.7]{Rezk01}. 
\end{example}

Recall that the category $\BS$ of bisimplicial sets admits a simplicial
enrichment, given by 
\[
\Map\pr{X,Y}=\BS\pr{\pr{\Delta^{0}\btimes\Delta^{\bullet}}\times X,Y}.
\]
The category $\BS^{+}$ admits a similar enrichment:
\begin{defn}
\label{def:BS^+_simplicial}We will regard $\BS^{+}$ as a simplicial
category as follows: If $\overline{X},\overline{Y}\in\SS^{+}$, then
the hom-simplicial set is given by
\[
\Map\pr{\overline{X},\overline{Y}}=\BS^{+}\pr{\pr{\pr{\Delta^{0}}^{\flat}\btimes\Delta^{\bullet}}\times\overline{X},\overline{Y}}.
\]
\end{defn}

\begin{rem}
The simplicial category $\BS^{+}$ is tensored and cotensored: If
$K$ is a simplicial set and $\overline{X}$ is a marked bisimplicial
set, their tensor $K\otimes\pr{X,A}$ is given by $\pr{\pr{\Delta^{0}}^{\flat}\boxtimes K}\times\overline{X}$
and their cotensor $\overline{X}^{K}$ is given by $\overline{X}^{\pr{\Delta^{0}}^{\flat}\btimes K}$.
\end{rem}

\begin{rem}
\label{rem:map_sat}Let $\cal C$ be a complete Segal space and let
$\overline{X}$ be a marked simplicial set. Then $\Map\pr{\overline{X},\cal C^{\natural}}\subset\Map\pr{X,\cal C}$
is the union of components corresponding to the maps $\overline{X}\to\cal C^{\natural}$
of marked bisimplicial sets. This is because $\cal C_{\hoeq}\subset\cal C_{1}$
is itself a union of components. In particular, every monomorphism
$\overline{X}\to\overline{Y}$ of marked bisimplicial sets induces
a Kan fibration 
\[
\Map\pr{\overline{Y},\cal C^{\natural}}\to\Map\pr{\overline{X},\cal C^{\natural}}.
\]
\end{rem}

\begin{defn}
A morphism $\overline{X}\to\overline{Y}$ of marked bisimplicial sets
is called a \textbf{marked equivalence} if for every complete Segal
space $\cal C$, the map
\[
\Map\pr{\overline{Y},\cal C^{\natural}}\to\Map\pr{\overline{X},\cal C^{\natural}}
\]
is a weak homotopy equivalence.
\end{defn}

We can now state the main Theorem of this section.
\begin{thm}
\label{thm:mCSS_model_structure}There is a combinatorial, simplicial,
cartesian model structure on $\BS^{+}$ which may be described as
follows:
\begin{enumerate}
\item The cofibrations are the monomorphisms.
\item The fibrat objects are the marked bisimplicial sets of the form $\cal C^{\natural}$,
where $\cal C$ is a complete Segal space.
\item Weak equivalences are the marked equivalences of marked bisimplicial
sets.
\item If $f:\cal C\to\cal D$ is a map of complete Segal spaces, then the
induced map $\cal C^{\natural}\to\cal D^{\natural}$ is a fibration
of $\BS_{\CSS}^{+}$ if and only if $f$ is a fibration of $\BS_{\CSS}$.
\item If $f:\cal C\to\cal D$ is a map of complete Segal spaces, then the
induced map $\cal C^{\natural}\to\cal D^{\natural}$ is a weak equivalence
of $\BS_{\CSS}^{+}$ if and only if $f$ induces homotopy equivalences
between the columns of $\cal C$ and $\cal D$.
\end{enumerate}
\end{thm}

The rest of this section is devoted to the proof of Theorem \ref{thm:mCSS_model_structure}.
We begin by defining a distinguished class of morphisms which will
be contained in the class of trivial cofibrations.
\begin{defn}
\label{def:mbe}The class of\textbf{ marked bianodyne extensions}
is the smallest saturated class of morphisms of marked bisimplicial
sets which contains the following morphisms:

\begin{enumerate}[label=(\Alph*)]

\item The inclusion $\pr{\partial\Delta^{n}\btimes\Delta^{m}}^{\flat}\cup\pr{\Delta^{n}\btimes\Lambda_{k}^{m}}^{\flat}\subset\pr{\Delta^{n}\btimes\Delta^{m}}^{\flat}$
for every $n\geq0$, $m\geq1$, and $0\leq k\leq m$.

\item The inclusion $\pr{\Lambda_{k}^{n}\btimes\Delta^{m}}^{\flat}\cup\pr{\Delta^{n}\btimes\partial\Delta^{m}}^{\flat}\subset\pr{\Delta^{n}\btimes\Delta^{m}}^{\flat}$
for every $0<k<n$ and $m\geq0$.

\item The inclusion $\pr{\{1\}\btimes\Delta^{m}}^{\flat}\cup\pr{J\btimes\partial\Delta^{m}}^{\flat}\subset\pr{J\btimes\Delta^{m}}^{\flat}$
for every $m\geq0$.

\item The inclusion $J^{\flat}\btimes\Delta^{m}\subset J^{\sharp}\btimes\Delta^{m}$
for every $m\geq0$.

\item The inclusion $\pr{\Delta^{1}}^{\sharp}\btimes\Delta^{m}\subset J^{\sharp}\btimes\Delta^{m}$
for every $m\ge0$.

\end{enumerate}
\end{defn}

We now look at basic properties of marked bianodyne extensions.
\begin{prop}
\label{prop:mbe_ext}Let $\overline{X}=\pr{X,S}$ be a marked bisimplicial
set. The following conditions are equivalent:
\begin{enumerate}
\item The marked bisimplicial set $\overline{X}$ has the right lifting
property for the marked bianodyne extensions.
\item The bisimplicial set $X$ is a complete Segal space, and $S$ is equal
to the sub Kan complex $X_{\hoeq}\subset X_{1}$ of homotopy equivalences.
\end{enumerate}
\end{prop}

\begin{proof}
The marked simplicial set $\overline{X}$ has the extension property
for the inclusions of types (A), (B), and (C) of Definition \ref{def:mbe}
if and only if $X$ is a complete Segal space. (See \cite[Propositions 2.5 and 3.4, Lemma 4.2]{JT07}.)
So conditions (1) and (2) both imply that $X$ is a complete Segal
space. We may therefore assume that $X$ is a complete Segal space.
In this case, the claim is that $\overline{X}$ has the extension
property for inclusions of types (D) and (E) of Definition \ref{def:mbe}
if and only $S=X_{\hoeq}$. To prove this, recall that the map
\[
\Map\pr{J\btimes\Delta^{0},X}\to\Map\pr{\Delta^{1}\btimes\Delta^{0},X}\cong X_{1}
\]
induces a trivial fibration $\Map\pr{J\btimes\Delta^{0},X}\to X_{\hoeq}$
(\cite[Theorem 6.2]{Rezk01}). It follows that a map $\Delta^{m}\to X_{1}$
of simplicial sets factors through $X_{\hoeq}$ if and only if its
adjoint $\Delta^{1}\btimes\Delta^{m}\to X$ extends to $J\btimes\Delta^{m}$.
(Said differently, $X_{\hoeq,m}$ is precisely the equivalences of
the $\infty$-category $X_{\ast,m}$.) Therefore, $\overline{X}$
has the extension property for inclusions of type (D) if and only
if $X_{\hoeq}\subset S$. Likewise, $\overline{X}$ has the extension
property for inclusions of type (E) if and only if $S\subset X_{\hoeq}$.
The claim follows.
\end{proof}
\begin{prop}
\label{prop:mbe->me}Every marked bianodyne extension is a marked
equivalence.
\end{prop}

\begin{proof}
Let $\cal C$ be a complete Segal space. Let $\scr C$ denote the
class of morphsims $\overline{X}\to\overline{Y}$ of marked bisimplicial
sets such that 
\[
\Map\pr{\overline{Y},\cal C^{\natural}}\to\Map\pr{\overline{X},\cal C^{\natural}}
\]
is a homotopy equivalence. We must show that $\scr C$ contains all
marked bianodyne extensions. The class $\scr C$ is saturated by Remark
\ref{rem:map_sat}, so it suffices to show that the morphisms of types
(A) through (E) of Definition \ref{def:mbe} are contained in $\scr C$.

\begin{enumerate}[label=(\Alph*)]

\item For each $n\geq0$, $m\geq1$, and $0\leq k\leq m$, the inclusion
$\pr{\partial\Delta^{n}\btimes\Delta^{m}\cup\Delta^{n}\btimes\Lambda_{k}^{m}}^{\flat}\subset\pr{\Delta^{n}\btimes\Delta^{m}}^{\flat}$
belongs to $\scr C$. Indeed, the map
\[
\Map\pr{\pr{\Delta^{n}\btimes\Delta^{m}}^{\flat},\cal C^{\natural}}\to\Map\pr{\pr{\partial\Delta^{n}\btimes\Delta^{m}\cup\Delta^{n}\btimes\Lambda_{k}^{m}}^{\flat},\cal C^{\natural}}
\]
can be identified with the map
\begin{align*}
\Fun\pr{\Delta^{m},\Map\pr{\Delta^{n}\btimes\Delta^{0},\cal C}}\\
\to\Fun\pr{\Delta^{m},\Map\pr{\partial\Delta^{n}\btimes\Delta^{0},\cal C}} & \times_{\Fun\pr{\Lambda_{k}^{m},\Map\pr{\partial\Delta^{n}\btimes\Delta^{0},\cal C}}}\Fun\pr{\Lambda_{k}^{m},\Map\pr{\Delta^{n}\btimes\Delta^{0},\cal C}}.
\end{align*}
This map is a trivial fibration because the map $\Map\pr{\Delta^{n}\btimes\Delta^{0},\cal C}\to\Map\pr{\partial\Delta^{n}\btimes\Delta^{0},\cal C}$
is a Kan fibration (\cite[Proposition 2.5]{JT07}) and the inclusion
$\Lambda_{k}^{m}\subset\Delta^{n}$ is anodyne.

\item For every $n\geq0$ and $0<k<m$, the inclusion $\pr{\Lambda_{k}^{n}\btimes\Delta^{m}\cup\Delta^{n}\btimes\partial\Delta^{m}}^{\flat}\subset\pr{\Delta^{n}\btimes\Delta^{m}}^{\flat}$
belongs to $\scr C$. This can be proved as in case (A), using the
fact that the map $\Map\pr{\Delta^{n}\btimes\Delta^{0},\cal C}\to\Map\pr{\Lambda_{k}^{n}\btimes\Delta^{0},\cal C}$
is a trivial fibration (\cite[Proposition 3.4]{JT07}).

\item For every $m\geq0$, the inclusion $\pr{\{1\}\btimes\Delta^{m}\cup J\btimes\partial\Delta^{m}}^{\flat}\subset\pr{J\btimes\Delta^{m}}^{\flat}$
belongs to $\scr C$. This can be proved as in case (A), using \cite[Proposition 6.4]{Rezk01}.

\item For each $m\geq0$, the map $J^{\flat}\btimes\Delta^{m}\to J^{\sharp}\btimes\Delta^{m}$
belongs to $\scr C$. Indeed, the map $\Map\pr{J^{\sharp}\btimes\Delta^{m},\cal C}\to\Map\pr{J^{\flat}\btimes\Delta^{m},\cal C}$
is an isomorphism of simplicial sets.

\item For each $m\geq0$, the map $\pr{\Delta^{1}}^{\sharp}\btimes\Delta^{m}\to J^{\sharp}\btimes\Delta^{m}$
belongs to $\scr C$. This follows from \cite[Theorem 6.2]{Rezk01},
which says that the map $\Map\pr{J^{\sharp}\btimes\Delta^{0},\cal C}\to\Map\pr{\pr{\Delta^{1}}^{\sharp}\btimes\Delta^{0},\cal C}$
is a trivial fibration.

\end{enumerate}
\end{proof}
\begin{prop}
\label{prop:marked_eq_fib}Let $\cal C$ and $\cal D$ be complete
Segal spaces and let $f:\cal C\to\cal D$ be a morphism of bisimplicial
sets. The following conditions are equivalent:
\begin{enumerate}
\item The map $f:\cal C^{\natural}\to\cal D^{\natural}$ is a marked equivalence
of marked bisimplicial sets.
\item The map $f:\cal C\to\cal D$ of bisimplicial sets is a weak equivalence
of $\BS_{\CSS}$.
\end{enumerate}
\end{prop}

\begin{proof}
If $\cal E$ is another complete Segal space, we have $\Map\pr{\cal C^{\natural},\cal E^{\natural}}=\Map\pr{\cal C,\cal E}$
and $\Map\pr{\cal D^{\natural},\cal E^{\natural}}=\Map\pr{\cal D,\cal E}$.
Therefore, the claim is a consequence of the fact that the complete
Segal space model structure on $\BS$ is simplicial \cite[Theorem 7.2]{Rezk01}. 
\end{proof}
We can now construct a candidate model structure for Theorem \ref{thm:mCSS_model_structure}.
\begin{thm}
\label{thm:marked_CSS_model_structure_1}There is a combinatorial
model structure on $\SS^{+}$ which satisfies the following conditions:
\begin{itemize}
\item The cofibrations are the monomorphisms.
\item The weak equivalences are the marked equivalences.
\end{itemize}
\end{thm}

\begin{proof}
Let $\scr W$ denote the class of marked equivalences and let $\scr C$
denote the class of monomorphisms of marked bisimplicial sets. By
\cite[Proposition A.2.6.8]{HTT}\footnote{This is A.2.6.10 in the latest version of Higher Topos Theory \cite{HTT17}.},
we must verify the following:
\begin{enumerate}
\item As a saturated class of morphisms of $\BS^{+}$, the class $\scr C$
is generated by a set of morphisms.
\item The class $\scr C\cap\scr W$ is saturated as a class of morphisms
of $\BS^{+}$.
\item The full subcategory $\bf W\subset\Fun\pr{[1],\BS^{+}}$ spanned by
$\scr W$ is accessible and the inclusion $\bf W\to\Fun\pr{[1],\BS^{+}}$
is an accessible functor.
\item The class $\scr W$ has the two out of three property.
\item If $f$ is a morphism of $\BS^{+}$ having the right lifting property
for all monomorphisms, then $f$ belongs to $\scr W$.
\end{enumerate}
For assertion (1), we simply observe that the inclusions $\{\pr{\partial\Delta^{n}\btimes\Delta^{m}\cup\Delta^{n}\btimes\partial\Delta^{m}}^{\flat}\subset\pr{\Delta^{n}\btimes\Delta^{m}}^{\flat}\mid n,m\ge0\}$
and $\{\pr{\Delta^{1}}^{\flat}\btimes\Delta^{n}\subset\pr{\Delta^{1}}^{\sharp}\btimes\Delta^{n}\mid n\geq0\}$
generate $\scr C$. Assertion (2) follows from Remark \ref{rem:map_sat}.
For assertion (3), use the small object argument \cite[Proposition A.1.2.5]{HTT}
to find a natural transformation $\alpha:\id_{\BS^{+}}\to T$ of endofunctors
of $\BS^{+}$ which satisfies the following conditions:

\begin{enumerate}[label=(\alph*)]

\item For each marked bisimplicial set $\overline{X}$, the map $\overline{X}\to T\overline{X}$
is marked bianodyne.

\item For each marked simplicial set $\overline{X}$, the marked
simplicial set $T\overline{X}$ has the extension property for all
marked bianodyne extensions.

\item The functor $T:\BS^{+}\to\BS^{+}$ commutes with $\kappa$-filtered
colimits, where $\kappa$ is some regular cardinal.

\end{enumerate}

By Propositions \ref{prop:mbe_ext} and \ref{prop:marked_eq_fib},
a morphism of $\BS^{+}$ is a weak equivalence if and only if the
composite $\BS^{+}\xrightarrow{T}\BS^{+}\xrightarrow{\Unm}\BS$ is
a weak equivalence in the complete Segal model structure. Since the
complete Segal model structure is combinatorial \cite[Proposition 9.1]{Rezk01},
assertion (3) now follows from \cite[Corollaries A.2.6.5 and A.2.6.9]{HTT}\footnote{Corollaries A.2.6.5 and A.2.6.9 correspond to A.2.6.7 and A.2.6.11
of \cite{HTT17}.}. 

Assertion (4) is obvious. For assertion (5), notice that a morphism
$f:\pr{X,S}\to\pr{Y,T}$ of marked bisimplicial sets has the right
lifting property for all monomorphisms of $\BS^{+}$ if and only if
it satisfies the following conditions:
\begin{itemize}
\item The map $X\to Y$ is a trivial fibration of $\BS_{\CSS}$.
\item The map $S_{n}\to T_{n}$ is surjective for every $n\geq0$. 
\end{itemize}
The claim is then immediate from Remark \ref{rem:map_sat}. 
\end{proof}
We will refer to the model structure of Theorem \ref{thm:marked_CSS_model_structure_1}
as the \textbf{marked complete Segal space model structure} and denote
it by $\BS_{\CSS}^{+}$. The rest of this section is devoted to showing
that this model structure has all the properties listed in Theorem
\ref{thm:mCSS_model_structure}.

\subsection{Identifying of Fibrant Objects}
\begin{prop}
\label{prop:mCSS_fibrant}The fibrant objects of $\BS_{\CSS}^{+}$
are the marked bisimplicial sets of the form $\cal C^{\natural}$,
where $\cal C$ is a complete Segal space.
\end{prop}

\begin{proof}
Every fibrant object of $\BS_{\CSS}^{+}$ necessarily has the extension
property for all marked biandoyne extensions (Proposition \ref{prop:mbe->me}),
so it has form $\cal C^{\natural}$ for some complete Segal space
$\cal C$ (Proposition \ref{prop:mbe_ext}). Conversely, if $\cal C$
is a complete Segal space, then $\cal C^{\natural}$ is fibrant by
the definition of marked equivalences and Remark \ref{rem:map_sat}.
The claim follows.
\end{proof}

\subsection{The Marked Complete Segal Space Model structure is Simplicial}
\begin{prop}
\label{prop:marked_CSS_simplicial}The marked complete Segal space
model structure is simplicial with respect to the simplicial enrichment
of Definition \ref{def:BS^+_simplicial}.
\end{prop}

\begin{proof}
Let $i:K\to L$ be a monomorphism of simplicial sets and let $j:\overline{X}\to\overline{Y}$
be a monomorphism of marked bisimplicial sets. We must show that the
map
\[
i\wedge j:\pr{K\otimes\overline{Y}}\amalg_{K\otimes\overline{X}}\pr{L\otimes\overline{X}}\to L\otimes\overline{Y}
\]
is a monomorphism, and that it is a marked equivalence if either $i$
is anodyne or $j$ is a marked equivalence. It is obvious that $i\wedge j$
is a monomorphism. Suppose that $i$ is anodyne or $j$ is a marked
equivalence. We must show that, for each complete Segal space $\cal C$,
the map
\[
\Map\pr{L\otimes\overline{Y},\cal C^{\natural}}\to\Map\pr{K\otimes\overline{Y},\cal C^{\natural}}\times_{\Map\pr{K\otimes\overline{X},\cal C^{\natural}}}\Map\pr{L\otimes\overline{X},\cal C^{\natural}}
\]
is a trivial fibration. We can identify this map with the map
\[
\Fun\pr{L,\Map\pr{\overline{Y},\cal C^{\natural}}}\to\Fun\pr{K,\Map\pr{\overline{Y},\cal C^{\natural}}}\times_{\Fun\pr{K,\Map\pr{\overline{X},\cal C^{\natural}}}}\Fun\pr{L,\Map\pr{\overline{X},\cal C^{\natural}}}.
\]
Using Remark \ref{rem:map_sat} and the fact that the Kan--Quillen
model structure on $\SS$ is cartesian, we deduce that the latter
map is a trivial fibration. 
\end{proof}

\subsection{Identifying Fibrations of Fibrant Objects}
\begin{prop}
\label{prop:bs_vs_bs^+}The adjunction
\[
\pr -^{\flat}:\BS_{\CSS}\adj\BS_{\CSS}^{+}:\Unm
\]
is a Quillen equivalence. 
\end{prop}

\begin{proof}
First we show that the adjunction is a Quillen adjunction. Since this
is a simplicial adjunction, it suffices to show that $\pr -^{\flat}$
preserves cofibrations and that $\Unm$ preserves fibrant objects
(\cite[Corollary A.3.7.2]{HTT}). The first assertion is obvious,
and the second follows from Proposition \ref{prop:mCSS_fibrant}.
To prove that it is a Quillen equivalence, let $\pr{\BS_{\CSS}}^{\circ}\subset\BS_{\CSS}$
and $\pr{\BS_{\CSS}^{+}}^{\circ}\subset\BS_{\CSS}^{+}$ denote the
full simplicial subcategory of fibrant objects. The simplicial functor
$\Unm:\pr{\BS_{\CSS}^{+}}^{\circ}\to\pr{\BS_{\CSS}}^{\circ}$ is an
isomorphism of simplicial categories, so it induces an isomorphism
of categories after considering the set of path components of each
hom simplicial sets. The latter functor models the total right derived
functor of $\Unm$, so this proves that $\Unm$ is a Quillen equivalence.
\end{proof}
\begin{prop}
\label{prop:mCSS_fibration}Let $f:\cal C\to\cal D$ be a map of complete
Segal spaces. Then the map $f:\cal C^{\natural}\to\cal D^{\natural}$
is a fibration of $\BS_{\CSS}^{+}$ if and only if the map $f:\cal C\to\cal D$
is a fibration of $\BS_{\CSS}$.
\end{prop}

\begin{proof}
By Proposition \ref{prop:bs_vs_bs^+}, every fibration of $\BS_{\CSS}^{+}$
becomes a fibration of $\BS_{\CSS}$ after forgetting the markings.
So necessity is obvious. For sufficiency, suppose that $f:\cal C\to\cal D$
is a fibration of $\BS_{\CSS}$. We must show that the induced map
$\cal C^{\natural}\to\cal D^{\natural}$ is a fibration of $\BS_{\CSS}^{+}$.
For this, it suffices to show that, for each trivial cofibration $\overline{X}\to\overline{Y}$
of $\BS_{\CSS}^{+}$, the map
\[
\theta:\Map\pr{\overline{Y},\cal C^{\natural}}\to\Map\pr{\overline{X},\cal D^{\natural}}\times_{\Map\pr{\overline{Y},\cal D^{\natural}}}\Map\pr{\overline{X},\cal C^{\natural}}
\]
is a trivial fibration. Since the model structure on $\BS_{\CSS}$
is simplicial, we deduce from Remark \ref{rem:map_sat} that $\theta$
is a Kan fibration. Therefore, it suffices to show that $\theta$
is a homotopy equivalence. For this, it suffices to show that the
maps
\begin{align*}
\Map\pr{\overline{Y},\cal C^{\natural}} & \to\Map\pr{\overline{X},\cal C^{\natural}},\\
\Map\pr{\overline{Y},\cal D^{\natural}} & \to\Map\pr{\overline{X},\cal D^{\natural}},
\end{align*}
are trivial fibrations. This follows from the fact that the model
structure on $\BS_{\CSS}^{+}$ is simplicial, and the proof is complete.
\end{proof}

\subsection{The Marked Complete Segal Space Model Structure is Cartesian}
\begin{defn}
Let $X$ be a bisimplicial set. Given a subset $S\subset X_{0,0}$,
we define the \textbf{full bisimplicial subset of $X$ spanned by
}$S$ to be the bisimplicial subset $Y\subset X$ consisting of those
$\pr{n,m}$-simplices $x\in X_{n,m}$ such that, for every morphism
$f:\pr{[0],[0]}\to\pr{[n],[m]}$ in $\Del\times\Del$, its pullback
$f^{*}x$ belongs to $S$. Said differently, $Y$ is obtained by first
taking the full simpliclial subset of $Y_{0}\subset X_{0}$ spanned
by $S$, and then taking the full simplicial subset of $X_{\ast,m}$
spanned by $Y_{0,m}$ for every $m\geq1$.
\end{defn}

\begin{prop}
Let $X$ be a bisimplicial set and let $Y\subset X$ be a full bisimplicial
subset.
\begin{enumerate}
\item If $X$ is vertically fibrant (i.e., fibrant in the vertical model
structure \cite[Theorem 2.6]{JT07}), so is $Y$.
\item If $X$ is a Segal space, so is $Y$.
\item If $X$ is a complete Segal space, so is $Y$.
\end{enumerate}
\end{prop}

\begin{proof}
First we prove (1). Suppose that $X$ is vertically fibrant. We must
show that $Y$ is vertically fibrant. In other words, we must show
that for each $n\ge0$, the map
\[
\Delta^{n}\backslash Y\to\partial\Delta^{n}\backslash Y
\]
is a Kan fibration. (See \cite[$\S$2]{JT07} for the definition of
the slice $-\backslash-$.) If $n=0$, this is clear because $Y_{0}$
is a Kan complex. If $n\geq1$, this follows from the observation
that, the map $\Delta^{n}\backslash Y\to\partial\Delta^{n}\backslash Y$
is a pullback of the Kan fibration $\Delta^{n}\backslash X\to\partial\Delta^{n}\backslash X$.
The proof of assertion (2) is similar. For assertion (3), suppose
that $X$ is a complete Segal space. We use \cite[Proposition 4.4]{JT07}
to prove that $Y$ is a complete Segal space: We must show that $Y$
satisfies the following conditions:

\begin{enumerate}[label=(\roman*)]

\item For each $n\geq0$, the map $Y/\Delta^{n}\to Y/\partial\Delta^{n}$
is a categorical fibration.

\item For each $n\geq0$, the map $\theta_{n}:Y/\Delta^{0}\to Y/\Delta^{n}$
is a categorical equivalence.

\end{enumerate}

Assertion (i) can be proved just like (1). For assertion (ii), we
will show that $\theta_{n}$ is fully faithful and essentially surjective.
Since $X$ is a complete Segal space, the functor $\theta'_{n}:X/\Delta^{0}\to X/\Delta^{n}$
is a categorical equivalence. It follows that $\theta_{n}$ is fully
faithful. To show that $\theta_{n}$ is essentially surjective, take
an arbitrary object $y\in Y/\Delta^{n}$. Find an object object $x\in X/\Delta^{0}$
such that $\theta'_{n}\pr x$ is equivalent to $y$ in $X/\Delta^{n}$.
Since the composite
\[
X/\Delta^{0}\xrightarrow{\theta'_{n}}X/\Delta^{n}\xrightarrow{\phi}X/\{0\}\cong X/\Delta^{0}
\]
is the identity map and $\phi$ carries $y$ into $Y/\Delta^{0}$,
we deduce that $x$ is equivalent to an object of $Y/\Delta^{0}$.
Hence $\theta_{n}$ is essentially surjective, and the proof is complete.
\end{proof}
\begin{prop}
\label{prop:pointwise_equivalence}Let $\cal C$ be a complete Segal
space and $X$ a bisimplicial set. A morphism of $\cal C^{X}$ is
an equivalence if and only if, for each $x\in X_{0,0}$, its image
under the map $\cal C^{X}\to\cal C^{\{x\}}$ is an equivalence.
\end{prop}

\begin{proof}
Necessity is obvious. For sufficiency, for each $p\geq0$, let $\sk_{p}X$
denote the $p$-skeleton of $X$. In other words, it is a bisimplicial
subset of $X$ consisting of those simplices $x\in A_{k,l}$ which
lies in the image of the map $\pr{f,g}^{*}:X_{i,j}\to X_{k,l}$, where
$f:[k]\to[i]$ and $g:[l]\to[j]$ are surjective poset maps such that
$i+j\leq p$. The bisimplicial set $X$ is the union of the increasing
sequence $\sk_{0}X\subset\sk_{1}X\subset\cdots$. We will show that,
for each $p\geq1$, the functor $\cal C^{\sk_{p}X}/\Delta^{0}\to\cal C^{\sk_{p-1}X}/\Delta^{0}$
is a conservative inner fibration of $\infty$-categories. Since conservative
inner fibrations are stable under transfinite cocompositions \cite[Corollary 2.1.10]{Landoo-cat},
it will then follow that the functor $\cal C^{X}/\Delta^{0}\to\cal C^{\sk_{0}X}/\Delta^{0}$
is conservative, proving the proposition.

For each $p\geq1$, there is a pushout square 
\[\begin{tikzcd}
	{\coprod_{\alpha\in\Sigma_p}((\partial\Delta^{n_\alpha}\boxtimes\Delta^{p-n_\alpha})\cup(\Delta^{n_\alpha}\boxtimes\partial\Delta^{p-n_\alpha}))} & {\operatorname{sk}_{p-1}X} \\
	{\coprod_{\alpha\in\Sigma_p}(\Delta^{n_\alpha}\boxtimes\Delta^{p-n_\alpha})} & {\operatorname{sk}_{p}X}
	\arrow[from=1-1, to=2-1]
	\arrow[from=1-1, to=1-2]
	\arrow[from=1-2, to=2-2]
	\arrow[from=2-1, to=2-2]
\end{tikzcd}\]for some set $\Sigma_{p}$ \cite[Theorem 1.3.8]{CisinskiHCHA}. Since
conservative inner fibrations are stable under pullback and products
\cite[Corollary 2.1.10]{Landoo-cat}, we are reduced to showing that,
for each $p\geq1$ and pair of integers $n,m\geq0$ such that $n+m=p$,
the functor
\[
\theta:\cal C^{\Delta^{n}\btimes\Delta^{m}}/\Delta^{0}\to\cal C^{\pr{\partial\Delta^{n}\btimes\Delta^{m}}\cup\pr{\Delta^{n}\btimes\partial\Delta^{m}}}/\Delta^{0}
\]
is a conservative inner fibration. It is an inner fibration, even
a categorical fibration, because the complete Segal model structure
is cartesian and the functor $-/\Delta^{0}:\BS_{\CSS}\to\SS_{{\rm Joyal}}$
is right Quillen \cite[4.11]{JT07}. To show that it is conservative,
we must consider several cases.
\begin{enumerate}
\item Suppose that $n\geq2$. By \cite[Proposition 3.10 and Lemma 7.14]{JT07},
the inclusion
\[
\pr{\Lambda_{1}^{n}\btimes\Delta^{m}}\cup\pr{\Delta^{n}\btimes\partial\Delta^{m}}\subset\Delta^{n}\btimes\Delta^{m}
\]
is a trivial cofibration of $\BS_{\CSS}$. It follows that the composite
\[
\cal C^{\Delta^{n}\btimes\Delta^{m}}/\Delta^{0}\xrightarrow{\theta}\cal C^{\pr{\partial\Delta^{n}\btimes\Delta^{m}}\cup\pr{\Delta^{n}\btimes\partial\Delta^{m}}}/\Delta^{0}\to\cal C^{\pr{\Lambda_{1}^{n}\btimes\Delta^{m}}\cup\pr{\Delta^{n}\btimes\partial\Delta^{m}}}/\Delta^{0}
\]
is a trivial fibration of $\infty$-categories. Hence $\theta$ is
conservative.
\item Suppose that $m\geq1$. By \cite[Proposition 2.5 and Lemma 7.14]{JT07},
the inclusion
\[
\pr{\partial\Delta^{n}\btimes\Delta^{m}}\cup\pr{\Delta^{n}\btimes\Lambda_{0}^{m}}\subset\Delta^{n}\btimes\Delta^{m}
\]
is a trivial cofibration of $\BS_{\CSS}$. It follows that the composite
\[
\cal C^{\Delta^{n}\btimes\Delta^{m}}/\Delta^{0}\xrightarrow{\theta}\cal C^{\pr{\partial\Delta^{n}\btimes\Delta^{m}}\cup\pr{\Delta^{n}\btimes\partial\Delta^{m}}}/\Delta^{0}\to\cal C^{\pr{\partial\Delta^{n}\btimes\Delta^{m}}\cup\pr{\Delta^{n}\btimes\Lambda_{0}^{m}}}/\Delta^{0}
\]
is a trivial fibration of $\infty$-categories. Hence $\theta$ is
conservative.
\item Suppose that $n=1$ and $m=0$. In this case, $\theta$ can be identified
with the functor
\[
\Fun\pr{\Delta^{1},X/\Delta^{0}}\to\Fun\pr{\partial\Delta^{1},X/\Delta^{0}},
\]
which is conservative by \cite[Proposition 2.2.3]{Landoo-cat}.
\end{enumerate}
\end{proof}
\begin{rem}
Proposition \ref{prop:pointwise_equivalence} remains valid in the
case where $\cal C$ is merely a Segal space, with the same proof.
\end{rem}

\begin{lem}
\label{lem:morphism_of_mbs}Let $\cal C$ be a complete Segal space
and $\overline{X}=\pr{X,S}$ a marked bisimplicial set. A map $\varphi:X\to\cal C$
of bisimplicial sets determines a map $\overline{X}\to\cal C^{\natural}$
of marked bisimplicial sets if and only if $\varphi$ carries each
element of $S_{0}$ to an equivalence of $\cal C$.
\end{lem}

\begin{proof}
This is immediate from the fact that the simplicial subset $\cal C_{\hoeq}\subset\cal C_{1}$
is a union of components.
\end{proof}
\begin{prop}
\label{prop:exponential_fibrant}Let $\cal C$ be a complete Segal
space and $\overline{X}$ a marked bisimplicial set. Then $\pr{\cal C^{\natural}}^{\overline{X}}=\cal D^{\natural}$,
where $\cal D\subset\cal C^{\overline{X}}$ denotes the full bisimplicial
subset spanned by the morphisms $\overline{X}\to\cal C^{\natural}$
of marked bisimplicial sets.
\end{prop}

\begin{proof}
Unwinding the definitions, we must prove the following:
\begin{enumerate}
\item Let $n,m\geq0$. A map $\pr{\Delta^{n}\btimes\Delta^{m}}\times X\to\cal C$
of bisimplicial sets determines a map $\pr{\Delta^{n}\btimes\Delta^{m}}^{\flat}\times\overline{X}\to\cal C^{\natural}$
if and only if it belongs to $\cal D$.
\item Let $m\geq0$. A map $\pr{\Delta^{1}\btimes\Delta^{m}}^{\flat}\times\overline{X}\to\cal C^{\natural}$
of marked bisimplicial sets determines a map $\pr{\Delta^{1}\btimes\Delta^{m}}^{\sharp}\times\overline{X}\to\cal C^{\natural}$
if and only if it belongs to $\cal D_{\hoeq}$.
\end{enumerate}
These assertions are consequences of Proposition \ref{prop:pointwise_equivalence}
and Lemma \ref{lem:morphism_of_mbs}.
\end{proof}
\begin{prop}
\label{prop:bss^+_cartesian}The marked complete Segal space model
structure is cartesian.
\end{prop}

\begin{proof}
We must show that for each pair of monomorphisms $i:\overline{A}\to\overline{B}$
and $j:\overline{X}\to\overline{Y}$ of marked bisimplicial sets,
the map
\[
\pr{\overline{A}\times\overline{Y}}\amalg_{\overline{A}\times\overline{X}}\pr{\overline{B}\times\overline{X}}\to\overline{B}\times\overline{Y}
\]
is a monomorphism, and that it is a marked equivalence if $i$ or
$j$ is a marked equivalence. The first assertion is obvious. For
the second assertion, it suffices to show that each object $\overline{A}\in\BS^{+}$
gives rise to a left Quillen functor $\overline{A}\times-:\BS_{\CSS}^{+}\to\BS_{\CSS}^{+}$.
Since it preserves cofibrations, it suffices to show that its right
adjoint $-^{\overline{A}}:\BS_{\CSS}^{+}\to\BS_{\CSS}^{+}$ carries
fibrations of fibrant objects to fibrations \cite[Proposition 7.15]{JT07}.
So let $f:\cal C^{\natural}\to\cal D^{\natural}$ be a fibration of
fibrant objects of $\BS_{\CSS}^{+}$. We wish to show that the map
$\pr{\cal C^{\natural}}^{\overline{A}}\to\pr{\cal D^{\natural}}^{\overline{A}}$
is a fibration of $\BS_{\CSS}^{+}$. According to Proposition \ref{prop:exponential_fibrant},
the marked bisimplicial sets $\pr{\cal C^{\natural}}^{\overline{A}}$
and $\pr{\cal D^{\natural}}^{\overline{A}}$ are fibrant in $\BS_{\CSS}^{+}$.
Thus, by Proposition \ref{prop:mCSS_fibration}, it suffices to show
that its underlying map is a fibration of $\BS_{\CSS}$. In other
words, we must show that for each $n\geq0$, $m\geq1$, and $0\le k\leq m$,
every lifting problem of the form
\[\begin{tikzcd}
	{((\partial\Delta^n\boxtimes\Delta^m)\cup (\Delta^n\boxtimes\Lambda^m_k))^\flat} & {(\mathcal{C}^\natural)^{\overline{A}}} \\
	{(\Delta^n\boxtimes \Delta^m)^\flat} & {(\mathcal{D}^\natural)^{\overline{A}}}
	\arrow[from=1-1, to=2-1]
	\arrow[from=2-1, to=2-2]
	\arrow[from=1-1, to=1-2]
	\arrow[from=1-2, to=2-2]
	\arrow[dashed, from=2-1, to=1-2]
\end{tikzcd}\]admits a solution. 

Suppose first that $n\geq1$ or $m\geq2$. In this case, the bisimplicial
subset $\pr{\partial\Delta^{n}\btimes\Delta^{m}}\cup\pr{\Delta^{n}\btimes\Lambda_{k}^{m}}\subset\Delta^{n}\btimes\Delta^{m}$
induces a bijection in the set of $\pr{0,0}$-simplices. Since $\Unm\pr{\pr{\cal C^{\natural}}^{\overline{A}}}$
and $\Unm\pr{\pr{\cal C^{\natural}}^{\overline{A}}}$ are full simplicial
subset of $\cal C^{A}$ and $\cal D^{A}$, and since $\mathcal{C}^{A}\to\mathcal{D}^{A}$
is a fibration of $\BS_{\CSS}$, it follows that the lifting problem
admits a solution. 

Next, suppose that $n=0$ and $m=1$. In this case, the lifting problem
we wish to solve can be rewritten as 
\[\begin{tikzcd}
	{(\Delta^0\boxtimes \{k\})^\flat\times \overline{A}} & {\mathcal{C}^\natural} \\
	{(\Delta^0\boxtimes \Delta^1)^\flat\times \overline{A}} & {\mathcal{D}^\natural.}
	\arrow[from=1-1, to=2-1]
	\arrow[from=1-2, to=2-2]
	\arrow["\phi"{description}, dashed, from=2-1, to=1-2]
	\arrow[from=2-1, to=2-2]
	\arrow[from=1-1, to=1-2]
\end{tikzcd}\]Since $\Unm\pr f$ is a fibration of fibrant objects of $\BS_{\CSS}$,
there is a map $\phi:\pr{\Delta^{0}\btimes\Delta^{1}}\times A\to\cal C$
rendering the diagram commutative if we forget the markings. We claim
that any such map in fact lifts to a morphism $\pr{\Delta^{0}\btimes\Delta^{1}}^{\flat}\times\overline{A}\to\cal C^{\natural}$. 

Let $\alpha\in A_{1,0}$ be an element which is marked in $\overline{A}$.
According to Lemma \ref{lem:morphism_of_mbs}, it suffices to show
that the morphism $\phi_{1,0}\pr{\id_{\varepsilon},\alpha}\in1,0$
is an equivalence of $\cal C$ for each $\varepsilon\in\{0,1\}$.
Here we identified the set $\pr{\Delta^{0}\btimes\Delta^{1}}_{1,0}\times A_{1,0}$
with $\Delta_{1,0}^{1}\times A_{1,0}$. Consider the element $\pr{0<1,s_{0}\alpha}\in A_{1,1}$.
The element $\phi_{1,1}\pr{0<1,s_{0}\alpha}\in\cal C_{1,1}$ determines
a path $\phi_{1,0}\pr{\id_{0},\alpha}\to\phi_{1,0}\pr{\id_{1},\alpha}$
in $\cal C_{1}$. Moreover, we know that one of $\phi_{1,0}\pr{\id_{0},\alpha}$
and $\phi_{1,0}\pr{\id_{1},\alpha}$ (namely, $\phi_{1,0}\pr{\id_{k},\alpha}$)
is an equivalence. Since $\cal C_{\hoeq}\subset\cal C_{1}$ is a union
of components, we deduce that both $\phi_{1,0}\pr{0,\alpha}$ and
$\phi_{1,0}\pr{1,\alpha}$ are equivalences. The claim follows.
\end{proof}

\subsection{Proof of Theorem \ref{thm:mCSS_model_structure}}
\begin{proof}
[Proof of Theorem \ref{thm:mCSS_model_structure}]We will show that
the marked complete Segal space model structure on $\BS^{+}$ has
the desired properties. By definition, it is combinatorial and satisfies
conditions (1) and (3). We saw that it is cartesian in Proposition
\ref{prop:bss^+_cartesian} and that it is simplicial in Proposition
\ref{prop:marked_CSS_simplicial}. Condition (2) follows from Proposition
\ref{prop:mCSS_fibrant}, and condition (4) follows from \ref{prop:mCSS_fibration}.
Condition (5) follows from Propositions \ref{prop:mCSS_fibrant} and
\ref{prop:marked_eq_fib}. The proof is now complete.
\end{proof}

\section{\label{sec:relations}Relating Various Model Structures}

In this section, we relate the marked complete Segal space model structure
with various model structures.

We start by introducing a pair of adjunctions relating marked simplicial
sets and marked bisimplicial sets, which are analogs of the unmarked
case \cite[Theorems 4.11 and 4.22]{JT07}.
\begin{defn}
\cite[$\S$2.2]{Rasekh21} Define a functor $\pr{p_{1}^{+}}^{*}:\SS^{+}\to\BS^{+}$
as follows: If $\pr{X,S}$ is a marked simplicial set, then $\pr{p_{1}^{+}}^{*}\pr{X,S}$
is the marked bisimplicial set whose rows are all $\pr{X,S}$. Define
also a functor $\pr{i_{1}^{+}}^{*}:\BS^{+}\to\SS^{+}$ by extracting
the first row. This defines an adjunction
\[
\pr{p_{1}^{+}}^{*}:\SS^{+}\adj\BS^{+}:\pr{i_{1}^{+}}^{*}.
\]
We also define a functor $\pr{t^{+}}_{!}:\BS^{+}\to\SS^{+}$ by
\[
\pr{t^{+}}_{!}\pr{X,S}=\pr{\opn{diag}X,S_{1}},
\]
where $\opn{diag}X$ denotes the simplicial set defined by $\pr{\opn{diag}X}_{n}=X_{n,n}$.
The right adjoint of $\pr{t^{+}}_{!}$ will be denoted by $\pr{t^{+}}^{!}$.
Explicitly, if $\overline{X}$ is a marked simplicial set, then $\pr{t^{+}}^{!}$
denotes the marked bisimplicial set whose underlying bisimplicial
set is given by $\pr{n,m}\mapsto\BS^{+}\pr{\pr{\Delta^{n}}^{\flat}\times\pr{\Delta^{m}}^{\sharp},\overline{X}}$.
The marked edges are the maps $\pr{\Delta^{1}}^{\sharp}\times\pr{\Delta^{m}}^{\sharp}\to\overline{X}$
of marked simplicial sets.
\end{defn}

We also need another model structure on $\BS^{+}$.
\begin{defn}
The category $\BS^{+}$ of marked bisimplicial sets can be identified
with the category $\pr{\SS^{+}}^{\Del^{\op}}$ of simplicial objects
in the category of marked simplicial sets. The \textbf{Reedy model
structure} on $\BS^{+}$, denoted by $\BS_{{\rm Reedy}}^{+}$, is
the Reedy model structure (equivalently, the injective model structure)
on $\pr{\SS^{+}}^{\Del^{\op}}$ with respect to the cartesian model
structure on $\SS_{\cart}^{+}$.
\end{defn}

We need yet another definition.
\begin{defn}
A marked bisimplicial set $\overline{X}=\pr{X,S}\in\BS^{+}$ is said
to be \textbf{categorically constant} if for each $n\geq0$, the map
$\pr{X_{\ast,0},S_{0}}\to\pr{X_{\ast,n},S_{n}}$ is a cartesian equivalence
of marked simplicial sets.
\end{defn}

Here is the main result of this section.
\begin{thm}
\label{thm:relations}
\begin{enumerate}
\item The adjunctions 
\[\begin{tikzcd}[sep=large]
	& {\mathsf{bsSet}^+_{\mathrm{Reedy}}} \\
	{\mathsf{sSet}^+_{\mathrm{cart}}} & {\mathsf{bsSet}^+_{\mathrm{CSS}}} & {\mathsf{sSet}^+_{\mathrm{cart}}} \\
	& {\mathsf{bsSet}_{\mathrm{CSS}}}
	\arrow["\dashv", shift right=3, from=1-2, to=2-2]
	\arrow["\bot"', shift left=3, from=2-3, to=2-2]
	\arrow["{(t^+)_!}", shift left=3, from=2-2, to=2-3]
	\arrow["{\operatorname{id}}"', shift right=3, from=2-2, to=1-2]
	\arrow["{(t^+)^!}", shift left=3, from=2-3, to=2-2]
	\arrow["{\operatorname{id}}"', shift right=3, from=1-2, to=2-2]
	\arrow["\bot"', shift left=3, from=2-2, to=2-1]
	\arrow["{(i_1^+)^*}", shift left=3, from=2-2, to=2-1]
	\arrow["{(p_1^+)^*}", shift left=3, from=2-1, to=2-2]
	\arrow["{(-)^\flat}", shift left=3, from=3-2, to=2-2]
	\arrow["\dashv"', shift left=3, from=3-2, to=2-2]
	\arrow["{\operatorname{Unm}}", shift left=3, from=2-2, to=3-2]
\end{tikzcd}\]are all Quillen adjunctions, and all but the top vertical adjunction
are Quillen equivalences.
\item The marked complete Segal space model structure on $\BS^{+}$ is a
Bousfield localization of the Reedy model structure. An object $\overline{X}\in\BS_{\CSS}^{+}$
is fibrant if and only if it is Reedy fibrant and categorically constant
(cf. \cite[Theorem 4.5]{JT07}).
\end{enumerate}
\end{thm}

The proof of Theorem \ref{thm:relations} will occupy the remainder
of this section.
\begin{prop}
\label{prop:cart_fib_vs_joy_fib}Let $p:\cal C\to\cal D$ be a functor
of $\infty$-categories. The following conditions are equivalent:
\begin{enumerate}
\item The map $\cal C^{\natural}\to\cal D^{\natural}$ of marked simplicial
sets is a fibration of $\SS_{\cart}^{+}$.
\item The functor $p$ is a categorical fibration.
\end{enumerate}
\end{prop}

\begin{proof}
Since the forgetful functor $\SS_{\cart}^{+}\to\SS_{{\rm Joyal}}$
is right Quillen (\cite[Proposition 3.1.5.3]{HTT}), the implication
(1)$\implies$(2) is obvious. The converse follows from \cite[Proposition 1.1.7]{JacoMT}
and \cite[Theorem 2.1.8]{Landoo-cat}.
\end{proof}
\begin{prop}
\label{prop:cart_vs_mCSS}The adjunctions
\begin{align*}
\pr{p_{1}^{+}}^{*} & :\SS_{{\rm cart}}^{+}\adj\BS_{\CSS}^{+}:\pr{i_{1}^{+}}^{*}\\
\pr{t^{+}}_{!} & :\BS_{\CSS}^{+}\adj\SS_{{\rm cart}}^{+}:\pr{t^{+}}^{!}
\end{align*}
are Quillen equivalences.
\end{prop}

\begin{proof}
First we show that the adjunctions are Quillen adjunctions. According
to \cite[Proposition 7.15]{JT07}, it suffices to prove the following:
\begin{enumerate}
\item The functor $\pr{p_{1}^{+}}^{*}$ preserves cofibrations.
\item The functor $\pr{i_{1}^{+}}^{*}$ preserves fibrations of fibrant
objects.
\item The functor $\pr{t^{+}}_{!}$ preserves cofibrations.
\item The functor $\pr{t^{+}}^{!}$ preserves fibrations of fibrant objects.
\end{enumerate}
Assertion (1) is obvious. For assertion (2), notice that the functor
$\pr{i_{1}^{+}}^{*}$ preserves fibrant objects. Therefore, by Proposition
\ref{prop:cart_fib_vs_joy_fib}, it suffices to show that the composite
\[
\BS_{\CSS}^{+}\xrightarrow{\pr{i_{1}^{+}}^{*}}\SS_{{\rm cart}}^{+}\xrightarrow{\text{forget}}\SS_{{\rm Joyal}}
\]
preserves fibrations. But this composite is even a right Quillen equivalence,
for it factors as a composite
\[
\BS_{\CSS}^{+}\xrightarrow{\Unm}\BS_{\CSS}\xrightarrow{i_{1}^{*}}\SS_{{\rm Joyal}}
\]
of right Quillen equivalences (Proposition \ref{prop:bs_vs_bs^+},
\cite[Proposition 4.11]{JT07}). This proves assertion (2).

Assertion (3) is obvious. For assertion (4), let $\cal C$ and $\cal D$
be $\infty$-categories and let $p:\cal C\to\cal D$ be a categorical
fibration. We must show that the map $\pr{t^{+}}^{!}\pr{\cal C^{\natural}}\to\pr{t^{+}}^{!}\pr{\cal D^{\natural}}$
is a fibration of $\BS_{\CSS}^{+}$. By Theorem \ref{thm:mCSS_model_structure},
it suffices to show that its underlying morphism of bisimplicial sets
is a fibration of $\BS_{\CSS}$. Unwinding the definitions, we must
show that, for every $n\geq0$, the map
\[
\Map^{\sharp}\pr{\pr{\Delta^{n}}^{\flat},\cal C^{\natural}}\to\Map^{\sharp}\pr{\pr{\partial\Delta^{n}}^{\flat},\cal C^{\natural}}\times_{\Map^{\sharp}\pr{\pr{\partial\Delta^{n}}^{\flat},\cal D^{\natural}}}\Map^{\sharp}\pr{\pr{\Delta^{n}}^{\flat},\cal D^{\natural}}
\]
is a Kan fibration. This follows from the fact that the cartesian
model structure is simplicial \cite[Corollary 3.1.4.4]{HTT}

We now complete the proof by showing that the adjunctions are Quillen
equivalences. Since the composite $\pr{i_{1}^{+}}^{*}\circ\pr{t^{+}}^{!}$
is the identity functor of $\SS^{+}$, it suffices to show that $\pr{i_{1}^{+}}^{*}$
is a Quillen equivalence. But this is clear from the proof of assertion
(2). The proof is now complete.
\end{proof}
\begin{prop}
\label{prop:Reedy_vs_mCSS}The adjunction
\[
\id:\BS_{\Reedy}^{+}\adj\BS_{\CSS}^{+}:\id
\]
is a Quillen adjunction.
\end{prop}

\begin{proof}
The model categories $\BS_{\Reedy}^{+}$ and $\BS_{\CSS}^{+}$ share
the same class of monomorphisms, so it suffices to show that every
weak equivalence of $\BS_{{\rm Reedy}}^{+}$ is a weak equivalence
of $\BS_{\CSS}^{+}$. Since every object of $\BS_{\CSS}^{+}$ is cofibrant
and the adjunction
\[
\pr{t^{+}}_{!}:\BS_{\CSS}^{+}\adj\SS_{\cart}^{+}:\pr{t^{+}}^{!}
\]
is a Quillen adjunction, the weak equivalences of $\BS_{\CSS}^{+}$
are precisely the morphisms whose images under $\pr{t^{+}}_{!}$ are
cartesian equivalences. It will therefore suffice to show that the
adjunction
\[
\pr{t^{+}}_{!}:\BS_{\Reedy}^{+}\adj\SS_{\cart}^{+}:\pr{t^{+}}^{!}
\]
is a Quillen adjunction. Clearly the functor $\pr{t^{+}}_{!}$ preserves
cofibrations. Therefore, by \cite[Proposition 7.15]{JT07}, we are
reduced to showing that $\pr{t^{+}}^{!}$ carries fibrations of fibrant
objects to Reedy fibrations.

Let $p:\cal C^{\natural}\to\cal D^{\natural}$ be a fibration of fibrant
objects of $\SS^{+}$. We must show that the map $\pr{t^{+}}^{!}\pr p:\pr{t^{+}}^{!}\pr{\cal C^{\natural}}\to\pr{t^{+}}^{!}\pr{\cal D^{\natural}}$
is a Reedy fibration. Unwinding the definitions, this amounts to showing
that, for each $m\geq0$, the map
\[
\pr{\cal C^{\natural}}^{\pr{\Delta^{m}}^{\sharp}}\to\pr{\cal D^{\natural}}^{\pr{\Delta^{m}}^{\sharp}}\times_{\pr{\cal D^{\natural}}^{\pr{\partial\Delta^{m}}^{\sharp}}}\pr{\cal C^{\natural}}^{\pr{\partial\Delta^{m}}^{\sharp}}
\]
is a fibration of $\SS_{\cart}^{+}$. This follows from the fact that
the cartesian model structure is cartesian \cite[Corollary 3.1.4.3]{HTT}.
\end{proof}
\begin{prop}
\label{prop:mCSS_fib_cat_const}A marked bisimplicial set is fibrant
in the marked complete Segal space structure if and only if it is
Reedy fibrant and categorically constant.
\end{prop}

\begin{proof}
Let $\overline{X}=\pr{X,S}$ be a marked bisimplicial set. Suppose
that $\overline{X}$ is Reedy fibrant and categorically constant.
Since Reedy fibrancy implies objectwise fibrancy, for each $m\geq0$,
the $m$th row $\pr{X_{\ast,m},S_{m}}\in\SS_{\cart}^{+}$ of $\overline{X}$
is fibrant. Moreover, the forgetful functor $\SS_{\cart}^{+}\to\SS_{{\rm Joyal}}$
is right Quillen \cite[Proposition 3.1.5.3]{HTT}, so it preserves
fibrations and weak equivalences of fibrant objects. It follows that
the bisimplicial set $X$ is fibrant in the horizontal model structure
of \cite[Proposition 2.10]{JT07} and moreover that it is categorically
constant in the sense of \cite[Definition 2.7]{JT07}. It follows
from \cite[Theorem 4.5]{JT07} that $X$ is a complete Segal space.
Since $S_{m}$ is the set of equivalences of the $\infty$-category
$X_{\ast,m}$, we find that $S=X_{\hoeq}$. Hence $\overline{X}$
is fibrant in the complete Segal space model structure. The reverse
implication can be proved similarly, using Proposition \ref{prop:cart_fib_vs_joy_fib}.
\end{proof}
\begin{proof}
[Proof of Theorem \ref{thm:relations}]Assertion (1) follows from
Propositions \ref{prop:bs_vs_bs^+}, \ref{prop:cart_vs_mCSS}, and
\ref{prop:Reedy_vs_mCSS}. Assertion (2) follows from assertion (1)
and Proposition \ref{prop:mCSS_fibrant}.
\end{proof}

\section{\label{sec:localization_thm}Main Result}

The goal of this section is to define the \textit{classification diagram
functor} $N:\SS^{+}\to\BS$ (Definition \ref{def:classification_diagram})
and show that it preserves and reflects weak equivalences (Theorem
\ref{thm:main}). We will also look at some of the consequences of
this.

We begin by generalizing Rezk's classification diagram to the setting
of marked simplicial sets.
\begin{defn}
\label{def:classification_diagram}Let $\overline{X}$ be a marked
simplicial set. Its \textbf{classification diagram} $N\pr{\overline{X}}$
is the underlying bisimplicial set of $\pr{t^{+}}^{!}\pr{\overline{X}}$.
The assignment $\overline{X}\mapsto N\pr{\overline{X}}$ defines a
functor $N:\SS^{+}\to\BS$.
\end{defn}

Here is the main result of this section.
\begin{thm}
\label{thm:main}Let $f:\overline{X}\to\overline{Y}$ be a morphism
of marked simplicial sets. The following conditions are equivalent:
\begin{enumerate}
\item The map $f$ is a weak equivalence of $\SS_{\cart}^{+}$.
\item The map $\pr{t^{+}}^{!}\pr f$ is a weak equivalence of $\BS_{\CSS}^{+}$.
\item The map $N\pr f$ is a weak equivalence of $\BS_{\CSS}$.
\end{enumerate}
\end{thm}

The main ingredients of the proof of Theorem \ref{thm:main} are Theorem
\ref{thm:relations}, an observation on the interaction of vertical
and horizontal directions of bisimplicial sets, and the next lemma.
\begin{lem}
\label{lem:simplicial_htpy}Let $\cal M$ be a simplicial model category,
let $X,Y\in\cal M$ be objects, and let $h:X\to Y^{\Delta^{1}}$ be
a morphism. For each $\varepsilon=0,1$, let $h_{\varepsilon}$ denote
the composite $h_{\varepsilon}:X\to Y^{\Delta^{1}}\to Y^{\{\varepsilon\}}\cong Y$.
The maps $h_{0}$ and $h_{1}$ represent the same morphism in the
homotopy category of the underlying model cateogry of $\cal M$.
\end{lem}

\begin{proof}
By choosing a weak equivalence $X'\to X$ with $X'$ cofibrant, we
may assume that $X$ is cofibrant. Let $k:X\otimes\Delta^{1}\to Y$
denote the adjoint of $h$, so that $k\vert X\otimes\{\varepsilon\}=h_{\varepsilon}$
for $\varepsilon=0,1$. Since $X$ is cofibrant, the definition of
simplicial model categories implies that the maps $X\otimes\{0\}\to X\otimes\Delta^{1}\ot X\otimes\{1\}$
are weak equivalences. Hence $h_{0}$ and $h_{1}$ represent the same
morphism in the underlying model category of $\cal M$.
\end{proof}
\begin{cor}
\label{cor:diagonal_weq}Let $\cal M$ be a simplicial model category
and let $X\in\cal M$ be an object. The map $i:X\to X^{\Delta^{n}}$
is a weak equivalence of $\cal M$.
\end{cor}

\begin{proof}
Let $r:X^{\Delta^{n}}\to X$ denote the retraction of $i$ induced
by the inclusion $\{0\}\subset\Delta^{n}$. We will show that the
map $ir:X^{\Delta^{n}}\to X^{\Delta^{n}}$ represents the same morphism
as the identity morphism in the homotopy category of the underlying
model category of $\cal M$. 

Define a map $H:\Delta^{n}\times\Delta^{1}\to\Delta^{n}$ by setting
$H\pr{i,0}=0$ and $H\pr{i,1}=i$ for each $0\leq i\leq n$. The claim
follows by applying Lemma \ref{lem:simplicial_htpy} to the map $h:X^{\Delta^{n}}\to X^{\Delta^{n}\times\Delta^{1}}$
induced by $H$.
\end{proof}
We can now prove Theorem \ref{thm:main}.
\begin{proof}
[Proof of Theorem \ref{thm:main}]The functors $\pr{t^{+}}^{!}$ and
$N=\Unm\circ\pr{t^{+}}^{!}$ are right Quillen equivalences (Theorem
\ref{thm:relations}), so they reflect weak equivalences of fibrant
objects. It follows that if these functors preserve weak equivalences,
then they reflect weak equivalences. It thus suffices to prove that
(1)$\implies$(2)$\implies$(3).

We begin with the proof of (1)$\implies$(2). Suppose that $f$ is
a cartesian equivalence. We must show that $\pr{t^{+}}^{!}\pr f$
is a weak equivalence of $\BS_{\CSS}^{+}$. By Theorem \ref{thm:relations},
it suffices to show that $\pr{t^{+}}^{!}\pr f$ is a Reedy weak equivalence.
In other words, it suffices to show that, for each $m\geq0$, the
map $\overline{X}^{\pr{\Delta^{m}}^{\sharp}}\to\overline{Y}^{\pr{\Delta^{m}}^{\sharp}}$
between the $m$th rows of $\overline{X}$ and $\overline{Y}$ is
a cartesian equivalence. This is immediate from Corollary \ref{cor:diagonal_weq}
and our assumption that $f$ is a cartesian equivalence.

Next we show that (2)$\implies$(3). Suppose that $\pr{t^{+}}^{!}\pr f$
is a weak equivalence. We must show that  $N\pr f$ is a weak equivalence.
In other words, we must show that, for each complete Segal space $\cal C$,
the map
\[
\Map\pr{N\pr{\overline{Y}},\cal C}\to\Map\pr{N\pr{\overline{X}},\cal C}
\]
is a homotopy equivalence of Kan complexes. Since $\pr{t^{+}}^{!}\pr f$
is a weak equivalence of $\BS_{\CSS}^{+}$, this will follow from
the following assertion:
\begin{itemize}
\item [($\ast$)]Let $\overline{A}=\pr{A,S}$ be a marked simplicial set.
The map $\Map\pr{\pr{t^{+}}^{!}\pr{\overline{A}},\cal C^{\natural}}\to\Map\pr{N\pr{\overline{A}},\cal C}$
is the identity map of simplicial sets.
\end{itemize}
To prove ($\ast$), recall that $\Map\pr{\pr{t^{+}}^{!}\pr{\overline{A}},\cal C^{\natural}}$
is a union of components of $\Map\pr{N\pr{\overline{A}},\cal C}$
(Remark \ref{rem:map_sat}). It thus suffices to show that the map
$\Map\pr{\pr{t^{+}}^{!}\pr{\overline{A}},\cal C^{\natural}}\to\Map\pr{N\pr{\overline{A}},\cal C}$
is surjective on vertices. 

Let $\varphi:N\pr{\overline{A}}\to\cal C$ be a map of bisimplicial
sets. We wish to show that $\varphi$ lifts to a map $\pr{t^{+}}^{!}\pr{\overline{A}}\to\cal C^{\natural}$
of marked bisimplicial sets. For this, it suffices to show that, for
each element $\alpha:a\to b$ of $S$, its image $\varphi_{1,0}\pr{\alpha}\in\cal{\cal C}_{1,0}$
is an equivalence of $\cal C$. Consider the diagram $\sigma:\Delta^{1}\times\Delta^{1}\to A$
which we depict as 
\[\begin{tikzcd}
	a & b \\
	b & b.
	\arrow["\alpha", from=1-1, to=1-2]
	\arrow["{1_b}"', from=2-1, to=2-2]
	\arrow["\alpha"', from=1-1, to=2-1]
	\arrow["{1_b}", from=1-2, to=2-2]
\end{tikzcd}\]Since $\alpha$ is marked, $\sigma$ determines an edge $\alpha\to1_{b}$
in the first column $N\pr{\overline{A}}_{1}$ of $N\pr{\overline{A}}$.
Its image $\varphi_{1,1}\pr{\sigma}\in\cal C_{1,1}$ is a path from
$\varphi_{1,0}\pr{\alpha}$ to $\varphi_{1,0}\pr{1_{b}}$ in the Kan
complex $\cal C_{1}$. Since $\varphi_{1,0}\pr{1_{b}}$ is an equivalence
of $\cal C$, so is $\varphi_{1,0}\pr{\alpha}$. The proof is now
complete.
\end{proof}
Let us look at two applications of Theorem \ref{thm:main}. As a first
application, we prove a generalization of Mazel-Gee's localization
theorem (Theorem \ref{thm:MGloc}). For this, we slightly extend our
definition of localizations of $\infty$-categories (Definition \ref{def:loc}).
\begin{defn}
Let $\pr{X,S}$ be a marked simplicial set and let $\cal D$ be an
$\infty$-category. A morphism $f:X\to\cal D$ of simplicial sets
is said to \textbf{exhibit $\cal D$ as a localization of $X$ with
respect to $S$} if it satisfies the following condition:
\begin{itemize}
\item [($\ast$)]For every $\infty$-category $\cal E$, the functor
\[
\Fun\pr{\cal D,\cal E}\to\Fun\pr{X,\cal E}
\]
is fully faithful, and its essential image consists of the diagrams
$X\to\cal E$ which carry every edge in $S$ to an equivalence.
\end{itemize}
\end{defn}

\begin{cor}
[A Generalization of Mazel-Gee's Localization Theorem]\label{cor:localization_theorem}Let
$\pr{X,S}$ be a marked simplicial set, let $\cal D$ be an $\infty$-category,
and let $f:\pr{X,S}\to\cal D^{\natural}$ be a morphism of marked
simplicial sets. The following conditions are equivalent:
\begin{enumerate}
\item The diagram $f:X\to\cal D$ exhibits $\cal D$ as a localization of
$X$ with respect to $S$.
\item The map $N\pr f:N\pr{X,S}\to N\pr{\cal D^{\natural}}$ is a weak equivalence
of $\BS_{\CSS}$.
\end{enumerate}
\end{cor}

\begin{proof}
Condition (1) is equivalent to the condition that the map $\pr{X,S}\to\cal D^{\natural}$
be a cartesian equivalence, so the claim is a consequence of Theorem
\ref{thm:main}.
\end{proof}
We now turn to the second application of Theorem \ref{thm:main},
which offers usable criteria to check that a certain functor is a
localization functor.
\begin{cor}
\label{cor:localization_by_column}Let $f:\overline{X}\to\overline{Y}$
be a morphism of  marked simplicial sets. Suppose that one of the
following conditions are satisfied:
\begin{enumerate}
\item For each $n\geq0$, the map $N\pr f:N\pr{\overline{X}}\to N\pr{\overline{Y}}$
induces a weak homotopy equivalence in the $n$th columns.
\item For each $m\geq0$, the map $N\pr f:N\pr{\overline{X}}\to N\pr{\overline{Y}}$
induces a weak categorical equivalence in the $m$th rows.
\end{enumerate}
Then $f$ is a cartesian equivalence of marked simplicial sets.
\end{cor}

\begin{proof}
Suppose that condition (1) holds. According to \cite[Theorem 2.6, 4.5]{JT07},
column-wise weak homotopy equivalences of bisimplicial sets are weak
equivalences in the complete Segal space model structure. It follows
from Theorem \ref{thm:main} that $f$ is a cartesian equivalence.
Likewise, since row-wise weak categorical equivalences are weak equivalences
in the complete Segal space model structure \cite[Theorem 4.5]{JT07},
condition (2) implies that $f$ is a cartesian equivalence.
\end{proof}
\providecommand{\bysame}{\leavevmode\hbox to3em{\hrulefill}\thinspace}
\providecommand{\MR}{\relax\ifhmode\unskip\space\fi MR }
\providecommand{\MRhref}[2]{%
  \href{http://www.ams.org/mathscinet-getitem?mr=#1}{#2}
}
\providecommand{\href}[2]{#2}

\end{document}